\documentclass[leqno]{amsart}

\usepackage{stmaryrd,graphicx}
\usepackage{amssymb,mathrsfs,amsmath,amscd,amsthm,color}
\usepackage{float}
\usepackage[all,cmtip]{xy}
\DeclareMathAlphabet{\mathpzc}{OT1}{pzc}{m}{it}
\usepackage{amsfonts,latexsym,wasysym}
\newcommand{\ui}{[0,1]}

\newcommand{\pionex}{\pi_{1}(X,x_0)}

\newcommand{\pioneh}{\pi_{1}(\bbh,b_0)}
\newcommand{\bpp}{b_{0}^{+}}
\newcommand{\lb}{\langle}
\newcommand{\rb}{\rangle}

\newcommand{\mcc}{\mathcal{C}}

\newcommand{\mci}{\mathcal{I}}

\newcommand{\scrl}{\mathscr{L}}

\newcommand{\bbw}{\mathbb{W}}

\newcommand{\bbh}{\mathbb{H}}
\newcommand{\bbn}{\mathbb{N}}
\newcommand{\bbq}{\mathbb{Q}}
\newcommand{\bbr}{\mathbb{R}}

\newcommand{\bbt}{\mathbb{T}}
\newcommand{\bbz}{\mathbb{Z}}

\newcommand{\bbhp}{\mathbb{H}^{+}}

\DeclareMathOperator{\cl}{Cl}
\DeclareMathOperator{\sct}{Sc}
\newcommand{\pionet}{\pi_1(\bbt,t_0)}
\newtheorem{theorem}{Theorem}[section]
\newtheorem{lemma}[theorem]{Lemma}
\newtheorem{proposition}[theorem]{Proposition}
\newtheorem{corollary}[theorem]{Corollary}
\theoremstyle{definition}\newtheorem{definition}[theorem]{Definition}
\newtheorem{example}[theorem]{Example}

\newtheorem{question}[theorem]{Question}
\newtheorem{remark}[theorem]{Remark}
\newtheorem{inductiveconstruct}[theorem]{The Inductive Construction of Closure}

\begin{document}
\title{Scattered products in fundamental groupoids}
\author{Jeremy Brazas}
\date{\today}

\begin{abstract}
Infinitary operations, such as products indexed by countably infinite linear orders, arise naturally in the context of fundamental groups and groupoids. We prove that the well-definedness of products indexed by a scattered linear order in the fundamental groupoid of a first countable space is equivalent to the homotopically Hausdorff property. To prove this characterization, we employ the machinery of closure operators, on the $\pi_1$-subgroup lattice, defined in terms of test maps from one-dimensional domains.
\end{abstract}

\maketitle

\section{Introduction and Results}

\subsection{Introduction}

The are many natural situations where a group $G$ with some additional structure admits an infinitary operation such as an infinite product, i.e. elements $\prod_{n=1}^{\infty}g_n\in G$ assigned to certain $\omega$-sequences $\{g_n\}\in G^{\omega}$, where product values satisfy $\prod_{n=1}^{\infty}g_n=(g_1g_2\cdots g_k)\prod_{n=k+1}^{\infty}g_{n}$ for all $k\in\bbn$. For example, the existence and uniqueness of classical infinite sums and products in $\bbr$ depends on the topology of $\bbr$. Non-abelian groups and groupoids with infinitary operations arise naturally in the context of ``wild" algebraic topology, in particular, the study of fundamental group(oid)s of spaces with non-trivial local homotopy. The well-definedness of such products in fundamental group(oid)s of one-dimensional and planar sets plays an important role in Katsuya Eda's homotopy classification of one-dimensional Peano continua \cite{Edaonedimhomtype} and related ``automatic continuity" results for fundamental groups of one-dimensional and planar Peano continua \cite{CK,Edafreessigmaplane,Kentplanarhomom}.

The fundamental group $\pi_1$ with its familiar binary operation is the skeleton of the fundamental groupoid $\Pi_1$ and, hence, determines it's composition operation. However, the determination of infinitary operations does not not pass so easily from groups to groupoids since infinitary products in $\pi_1(X,x)$, $x\in X$ depend on point-local structure and only partially inform infinitary products in $\Pi_1(X)$, which may depend more on the global structure of $X$.

In the context of fundamental groups (resp. groupoids), infinite product values are defined in terms of the topology of representing loops (resp. paths): if $\alpha_n$ is a sequence of loops (resp. paths), then the $\omega$-product $\prod_{n=1}^{\infty}[\alpha_n]$ is the homotopy class of the infinite concatenation $(\alpha_1\cdot (\alpha_2\cdot (\alpha_3\cdot (\cdots ))))$ when the concatenation is defined and continuous. The $\omega$-product operation on homotopy classes is well-defined if path-homotopic factors result in path-homotopic infinite concatenations. There are subspaces of $\bbr^3$ for which this is not the case (see Example \ref{archipelagoexample}).

Since the components of an open set of $\ui$ may have any countable linear order type, one may also define group(oid) products $\prod_{i\in \scrl}[\alpha_i]$ indexed by any countable linear order $\scrl$. The primary purpose of this paper is to characterize those spaces whose fundamental groups and/or groupoids admit well-defined $\omega$-products and, more generally, well-defined products over scattered linear orders.

\subsection{Main Results}

To effectively formalize the well-definedness of products indexed over linear orders, we employ the fundamental group of the usual Hawaiian earring space $\bbh$ and the loop $\ell_n:(\ui,\{0,1\})\to(\bbh,b_0)$ traversing the n-th circle. The group $\pioneh$ is isomorphic to a locally free group of reduced words where the letters of each word are indexed by a countable linear order and each letter $[\ell_n]$ and it's inverse may appear at most finitely many times in each word \cite{CChegroup,Edafreesigmaproducts}. The \textit{subgroup of scattered words} $Sc\leq \pioneh$ consists of homotopy classes of loops $\alpha:(\ui,\{0,1\})\to (\bbh,b_0)$ for which $\alpha^{-1}(b_0)$ is a scattered set, or equivalently, for which the components of $\ui\backslash \alpha^{-1}(b_0)$ have a scattered order type (see  \cite{EK99subgroups}). In what follows, $\pi_1$ and $\Pi_1$ refer to the fundamental group and groupoid respectively.

\begin{definition}\label{maindef}
Let $X$ be a path-connected space. We say that $X$ has
\begin{enumerate}
\item \textit{well-defined infinite $\pi_1$-products} if for any $x\in X$ and maps $f,g:(\bbh,b_0)\to (X,x)$ such that the induced homomorphisms $f_{\#},g_{\#}:\pioneh\to\pi_1(X,x)$ agree each individual letters, i.e. $f_{\#}([\ell_n])=g_{\#}([\ell_n])$ for all $n\in\bbn$, then $f_{\#}$ and $g_{\#}$ also agree on the standard infinite product $[\ell_1\cdot(\ell_2\cdot(\ell_3\cdot (\cdots)))]$.
\item \textit{well-defined scattered $\pi_1$-products} if for any $x\in X$ and maps $f,g:(\bbh,b_0)\to (X,x)$ such that $f_{\#}([\ell_n])=g_{\#}([\ell_n])$ for all $n\in\bbn$, then $f_{\#}$ and $g_{\#}$ also agree on the subgroup $Sc\leq \pioneh$ of scattered words.
\item \textit{well-defined infinite $\Pi_1$-products} if for any paths $\alpha,\beta:\ui\to X$ that agree on $S= \{\frac{n-1}{n}\mid n\in\bbn\}\cup\{1\}$ and such that $\alpha|_{[a,b]}$ is path homotopic to $\beta|_{[a,b]}$ for each component $(a,b)$ of $\ui\backslash S$, then $[\alpha]=[\beta]$ in $\Pi_1(X)$.
\item \textit{well-defined scattered $\Pi_1$-products} if for any paths $\alpha,\beta:\ui\to X$ that agree on a closed scattered set $S\subset\ui$ containing $\{0,1\}$ and such that $\alpha|_{[a,b]}$ is path homotopic to $\beta|_{[a,b]}$ for each component $(a,b)$ of $\ui\backslash S$, then $[\alpha]=[\beta]$ in $\Pi_1(X)$.
\end{enumerate}
\end{definition}

The \textit{homotopically Hausdorff} property (see Definition \ref{homhausdorffdef}) is the most fundamental point-local obstruction to applying shape theoretic methods \cite{CF,EK98,FZ05} and generalized covering space methods \cite{Brazcat,BDLM08,FZ07,FZ013caley} to study fundamental groups of spaces that lack a simply connected covering space. Spaces known to be homotopically Hausdorff include those spaces whose fundamental group naturally injects into the first shape homotopy group \cite{FZ07}, including all one-dimensional spaces, planar sets, and certain inverse limits of manifolds. Additionally, many other examples exhibiting more intricate infinite $\pi_1$-product operations are homotopically Hausdorff, e.g. $A,B\subset\bbr^3$ in \cite{CMRZZ08}, the space $RX$ in \cite{VZ13}, and the ``Hawaiian pants" space $\mathbb{P}\subset\bbr^3$ in \cite{BFThawaiianpants}. We refer to \cite{BFTestMap,CMRZZ08,FRVZ11} for characterizations and comparisons of the homotopically Hausdorff property with other local properties. The main result of this paper is the following theorem. 

\begin{theorem}\label{overallthm}
For any path-connected first countable space $X$, the following are equivalent:
\begin{enumerate}
\item $X$ has well-defined infinite $\pi_1$-products,
\item $X$ has well-defined scattered $\pi_1$-products, 
\item $X$ has well-defined infinite $\Pi_1$-products,
\item $X$ has well-defined scattered $\Pi_1$-products,
\item $X$ is homotopically Hausdorff.
\end{enumerate}
\end{theorem}

Obviously, (4) $\Rightarrow$ (3) $\Rightarrow$ (1) and (4) $\Rightarrow$ (2) $\Rightarrow$ (1) since loops are paths and $\omega$ is a scattered order. The primary difficulty in the proof of Theorem \ref{overallthm} is to verify the equivalence of (2) and (4) with the other properties. In fact, we prove a subgroup-relative version of (4) $\Leftrightarrow$ (5), which applies to arbitrary subgroups of $\pionex$ (see Theorem \ref{mainspathproducts}). To achieve this, we utilize the closure operator framework introduced in \cite{BFTestMap}. In particular, we compute subgroup closures (see Theorem \ref{onedtheorem} and Corollary \ref{agree}) by applying the theory of reduced paths in one-dimensional spaces \cite{CConedim} and adapting the proof of Hausdorff's characterization of scattered linear orders (see Chapter 5 of \cite{RosensteinLO}).

The remainder of this paper is structured as follows: In Section 2, we set up notation and recall relevant background on paths and linear orders. In Section 3, we recall the closure operators on the $\pi_1$-subgroup lattice from \cite{BFTestMap} and introduce an inductive construction of these closures. In Section 4, we introduce the \textit{scattered extension} $\sct(H)$ of a subgroup $H\leq \pionex$ and relate it to closure under infinite products (see Theorem \ref{mainthm}). In Section 5, we recall the relative homotopically Hausdorff property and use the scattered extension to prove the equivalence (1) $\Leftrightarrow$ (2) $\Leftrightarrow$ (5) in Theorem \ref{overallthm} (see Theorem \ref{groupthm}). In Section 6, we introduce a closure operator that allows us to establish the equivalence of (3) and (4) with the other three properties in Theorem \ref{overallthm} (see Remark \ref{lastremark}). We conclude the paper, in Section 7, with a brief remark on products indexed by arbitrary countable order types.

\section{preliminaries and notation}

Throughout this paper, $X$ will denote a path-connected topological space and $x_0\in X$ will be a basepoint. The homomorphism induced on $\pi_1$ by a based map $f:(X,x)\to (Y,y)$ is denoted $f_{\#}:\pi_1(X,x)\to\pi_1(Y,y)$.

A \textit{path} is a continuous function $\alpha:[0,1]\to X$, which we call a \textit{loop} based at $x\in X$ if $\alpha(0)=\alpha(1)=x$. If $[a,b],[c,d]\subseteq \ui$ and $\alpha:[a,b]\to X$, $\beta:[c,d]\to X$ are maps, we write $\alpha\equiv\beta$ if $\alpha=\beta\circ \phi$ for some increasing homeomorphism $\phi: [a,b]\to [c,d]$; if $\phi$ is linear and if it does not create confusion, we will identify $\alpha$ and $\beta$. Under this identification, the restriction $\alpha|_{[a,b]}$ of a path $\alpha:\ui\to X$ is a path itself with a path-homotopy class $[\alpha|_{[a,b]}]$.

If $\alpha:\ui\to X$ is a path, then $\alpha^{-}(t)=\alpha(1-t)$ is the reverse path. If $\alpha_1,\alpha_2,\dots,\alpha_n$ is a sequence of paths such that $\alpha_{j}(1)=\alpha_{j+1}(0)$ for each $j$, then $\prod_{j=1}^{n}\alpha_j=\alpha_1\cdot \alpha_2\cdot\;\cdots\;\cdot \alpha_n$ is the path defined as $\alpha_j$ on $\left[\frac{j-1}{n},\frac{j}{n}\right]$. If $\alpha_1,\alpha_2,\alpha_3,\dots$ is an infinite sequence of paths in $X$ such that $\alpha_n(1)=\alpha_{n+1}(0)$ for all $n\in\bbn$ and there is a point $x\in X$ such that every neighborhood of $x$ contains $\alpha_n([0,1])$ for all but finitely many $n$, then the \textit{infinite concatenation} of this sequence is the path $\alpha=\prod_{n=1}^{\infty}\alpha_n=(\alpha_1\cdot(\alpha_2\cdot(\alpha_3\cdot (\cdots))))$ defined to be $\alpha_n$ on $\left[\frac{n-1}{n},\frac{n}{n+1}\right]$ and $\alpha(1)=x$.

A path $\alpha:[a,b]\to X$ is \textit{reduced} if $\alpha$ is constant or if whenever $a\leq s<t\leq b$ with $\alpha(s)=\alpha(t)$, the loop $\alpha|_{[s,t]}$ is not null-homotopic. If $X$ is a one-dimensional metric space, then every path $\alpha:\ui\to X$ is path homotopic within the image of $\alpha$ to a reduced path, which is unique up to reparameterization \cite{EdaSpatial}.

For basic theory of linearly ordered sets, we refer to \cite{RosensteinLO}. Let $\mathbf{n}$, $\omega$, and $\zeta$ denote the order types of the n-point set, the natural numbers, and $\bbz$ respectively. If $(L,\leq)$ is a linearly ordered set, let $L^{\ast}$ denote $L$ with the reverse ordering $\leq^{\ast}$ where $b\leq^{\ast} a$ if and only if $a\leq b$. If $A,B\subseteq L$, then we write $A\leq B$ if $a\leq b$ for all $a\in A$ and $b\in B$. 

\begin{definition}\label{lodef}
Let $(L,\leq)$ be a linearly ordered set.
\begin{enumerate}
\item $L$ is \textit{dense} if $L$ has more than one point and if for each $x,y\in L$ with $x<y$, there exists $z\in L$ with $x<z<y$.
\item $L$ is a \textit{scattered order} if $L$ contains no dense suborders,
\item A \textit{cut} of $L$ is a pair $(A,B)$ of disjoint sets whose union is $L$ and $A<B$. The trivial cuts are the pairs where either $A=\emptyset$ or $B=\emptyset$. Let $\mcc(L)$ denote the set of cuts of $L$ with its natural linear ordering: $(A,B)\leq (A',B')$ if $A\subseteq A'$.
\end{enumerate}
\end{definition} 

Every countable linear order embeds as a suborder of the dense ordered set $\bbq$.

\begin{definition}
A topological space $X$ is a \textit{scattered space} if every non-empty subspace of $X$ contains an isolated point.
\end{definition}

Every scattered order is a scattered space with the order topology. However, $\bbq\times\bbz$ with the lexicographical ordering is discrete but is not a scattered order since it contains a dense suborder. Generally, our use of the word ``scattered" will be clear from context, however, if confusion is possible, we will clarify by specifically stating ``scattered order" or ``scattered space." 

The following lemma is a combination of well-known results in linear order theory \cite{RosensteinLO} and descriptive set theory \cite{KechrisDST} and will be used implicitly throughout the paper.

\begin{lemma}\label{scatteredordercharlemma}
For any countable linear order $L$, the following are equivalent:
\begin{enumerate}
\item $L$ is a scattered order,
\item $\mcc(L)$ embeds as a countable compact subset of $\bbr$,
\item $\mcc(L)$ is a scattered space with the natural order topology.
\end{enumerate}
\end{lemma}
\begin{definition}
If $K$ is a non-degenerate compact subset of $\bbr$, let $\mci(K)$ denote the set of components of $[\min(K),\max(K)]\backslash K$ equipped with the ordering inherited from $\bbr$. 
\end{definition}
Note that $\mci(K)$ is always countable and if $K$ is nowhere dense (for example, if $K$ is a scattered space), then $\mcc(\mci(K))\cong K$ as spaces. On the other hand, for any countable linear order $L$ and closed interval $[a,b]$ in $\bbr$, we may identify $\mathcal{C}(L)$ with a subspace of $[a,b]$ so that $a=\min(\mathcal{C}(L))$ and $b=\max(\mathcal{C}(L))$. Any such choice of embedding $\mathcal{C}(L)\subset[a,b]$ determines an order isomorphism $\mci(\mathcal{C}(L))\cong L$.

\section{closure operators on subgroups of fundamental groups}

The following definitions are from \cite{BFTestMap} where closure operators of subgroups are introduced to characterize and compare local properties of fundamental groups.

\begin{definition}\label{testmap}
Suppose $(\bbt,t_0)$ is a based space, $T\leq \pionet$ is a subgroup, and $g\in \pionet$. A subgroup $H\leq \pionex$ is $(T,g)${\it-closed} if for every based map $f:(\bbt,t_0)\to (X,x_0)$ such that $f_{\#}(T)\leq H$, we also have $f_{\#}(g)\in H$. We refer to $(T,g)$ as a {\it closure pair} for $(\bbt,t_0)$.
\end{definition}
The set of $(T,g)$-closed subgroups of $\pionex$ is closed under intersection and therefore forms a complete lattice.

\begin{definition}
The $(T,g)${\it-closure} of a subgroup $H\leq \pionex$ is \[ \cl_{T,g}(H)=\bigcap\{K\leq \pionex\mid K\text{ is }(T,g)\text{-closed and }H\leq K\}.\]
\end{definition}

\begin{lemma}[Closure Operator Properties of $\cl_{T,g}$]{\cite[Section 2]{BFTestMap}}\label{closurepropertieslemma}
Let $(T,g)$ be a closure pair. Then $\cl_{T,g}(H)=H$ if and only if $H$ is $(T,g)$-closed. Moreover,
\begin{enumerate}
\item $H\leq \cl_{T,g}(H)$,
\item $H\leq K$ implies $\cl_{T,g}(H)\leq \cl_{T,g}(K)$,
\item $\cl_{T,g}(\cl_{T,g}(H))=\cl_{T,g}(H)$,
\item if $f:(X,x_0)\to (Y,y_0)$ is a map, then $f_{\#}(\cl_{T,g}(H))\leq \cl_{T,g}(f_{\#}(H))$ in $\pi_1(Y,y_0)$.
\end{enumerate}
\end{lemma}

The closure $\cl_{T,g}(H)$ must contain the subgroup of $\pionex$ generated by $H$ and the set of elements $f_{\#}(g)$ for all maps $f:(T,t_0)\to (X,x_0)$ with $f_{\#}(T)\leq H$. However, it is shown \cite[Remark 3.13]{BFTestMap} that $\langle H\cup \{f_{\#}(g)\mid f:(T,t_0)\to (X,x_0)\text{ with }f_{\#}(T)\leq H\}\rangle$ may not be $(T,g)$-closed and therefore may be a proper subgroup of $\cl_{T,g}(H)$. In general, we must use the following inductive procedure to construct $\cl_{T,g}(H)$ from $H$.

\begin{inductiveconstruct}
Given a closure pair $(T,g)$ and subgroup $H\leq \pionex$, set $\cl_{T,g}(H)_{\mathbf{0}}=H$ and suppose $\cl_{T,g}(H)_{\lambda}$ is defined for all ordinals $\lambda<\kappa$. If $\kappa=\lambda+1$ is a successor, let $\cl_{T,g}(H)_{\kappa}$ be the subgroup of $\pionex$ generated by $\cl_{T,g}(H)_{\lambda}$ and the elements $f_{\#}(g)$ for all maps $f:(\bbt,t_0)\to (X,x_0)$ satisfying $f_{\#}(T)\leq \cl_{T,g}(H)_{\lambda}$. If $\kappa$ is a limit ordinal, let $\cl_{T,g}(H)_{\kappa}=\bigcup_{\lambda<\kappa}\cl_{T,g}(H)_{\lambda}$, which is also a subgroup of $\pionex$.

Note that $\cl_{T,g}(H)_{\kappa}$ is $(T,g)$-closed if and only if $\cl_{T,g}(H)_{\kappa}=\cl_{T,g}(H)_{\kappa+1}$. Hence, by basic cardinal considerations, there must be some smallest ordinal $\kappa_{0}$ such that $\cl_{T,g}(H)_{\kappa}=\cl_{T,g}(H)_{\kappa_0}$ for all $\kappa\geq \kappa_{0}$. By transfinite induction, every $(T,g)$-closed subgroup containing $H$ also contains $\cl_{T,g}(H)_{\kappa}$ for each $\kappa$. Therefore, we have $\cl_{T,g}(H)=\cl_{T,g}(H)_{\kappa_0}=\bigcup_{\kappa}\cl_{T,g}(H)_{\kappa}$.
\end{inductiveconstruct}

\begin{definition}
The $(T,g)$\textit{-rank} of an element $a\in \cl_{T,g}(H)$, is the smallest ordinal $\kappa_0$ such that $a\in \cl_{T,g}(H)_{\kappa_0}$.
\end{definition}

We compare closure operators using the following remark.

\begin{remark}\cite[Proposition 2.3]{BFTestMap}\label{comparisonremark}
If $(T,g)$ and $(T',g')$ are any closure pairs for a test space $(\bbt,t_0)$, then $g'\in \cl_{T,g}(T')$ if and only if $\cl_{T',g'}(H)\leq \cl_{T,g}(H)$ for all spaces $(X,x_0)$ and subgroups $H\leq \pionex$.
\end{remark}

Let $C_n\subseteq \bbr^2$ be the circle of radius $\frac{1}{n}$ centered at $\left(\frac{1}{n},0\right)$ and $\bbh=\bigcup_{n\in\bbn}C_n$ be the usual Hawaiian earring space with basepoint $b_0=(0,0)$. Let $\ell_n(t)=\left(\frac{1}{n}(1-\cos(2\pi t)),-\frac{1}{n}\sin(2\pi t)\right)$ be the canonical counterclockwise loop traversing $C_n$ with homotopy class $a_n=[\ell_n]$. These elements freely generate the subgroup $F=\lb a_n\mid n\in\bbn\rb\leq\pi_1(\bbh,b_0)$. Let $\ell_{\infty}$ denote the infinite concatenation $\prod_{n=1}^{\infty}\ell_n$ and $a_{\infty}=[\ell_{\infty}]$.

We will consider the closure operator $\cl_{F,a_{\infty}}$ in the following section. Recall that this closure is constructed only using based maps $(\bbh,b_0)\to (X,x_0)$ and, therefore, only describes local homotopy structure at $x_0$. 

\section{The scattered extension of a subgroup}

\begin{definition}\label{schdef}
Let $H\leq \pionex$ be a subgroup. A non-constant loop $\alpha:\ui\to X$ based at $x_0$ is $H$\textit{-scattered} if there exists a closed scattered space $K\subseteq \alpha^{-1}(x_0)$ containing $\{0,1\}$ such that for every component $(a,b)\in \mci(K)$, we have $[\alpha|_{[a,b]}]\in H$. We call the set $K$ a set of $H$\textit{-cuts} for $\alpha$. The \textit{scattered extension} of $H$ is \[\sct(H)=\{[\alpha]\in\pionex\mid\alpha\text{ is }H\text{-scattered or constant}\}.\]
\end{definition}

Since the union of two closed scattered subspaces of $\ui$ is a scattered space, $\sct(H)$ is a subgroup of $\pionex$ containing $H$.

\begin{lemma}\label{sctproperties}
As an operator on subgroups of $\pionex$, $\sct$ satisfies the following:
\begin{enumerate}
\item $H\leq \sct(H)$,
\item $H\leq K$ $\Rightarrow$ $\sct(H)\leq \sct(K)$,
\item $\sct(H)\leq \sct(\sct(H))$,
\item if $H\leq \pionex$ and $f:(X,x_0)\to (Y,y_0)$ is a map, then $f_{\#}(\sct(H))\leq \sct(f_{\#}(H))$.
\end{enumerate}
\end{lemma}

\begin{proof}
(1) and (2) are straightforward to verify and (3) follows from (1) and (2). To verify (4), suppose $[\alpha]\in\sct(H)$ and $S\subseteq \ui$ is a set of $H$-cuts for representative $H$-scattered loop $\alpha$. For all $(a,b)\in\mci(S)$, we have $[\alpha|_{[a,b]}]\in H$ and, thus, $[f\circ \alpha|_{[a,b]}]\in f_{\#}(H)$. It follows that $S$ is a set of $f_{\#}(H)$-cuts for $f\circ\alpha$. Hence, $f_{\#}([\alpha])\in \sct(f_{\#}(H))$.
\end{proof}

\begin{remark}
Compare the previous lemma with Lemma \ref{closurepropertieslemma}, particularly the difference of property (3). We refrain from calling $\sct$ a ``closure" since, as an operator on subgroups, $\sct$ behaves more like the first step $\cl_{F,a_{\infty}}(-)_{\mathbf{1}}$ in the inductive construction of $\cl_{F,a_{\infty}}$. One should not expect $\sct(\sct(H))=\sct(H)$ to hold in general.
\end{remark}

For $H\leq \pionex$, each non-trivial element of $\sct(H)$ is represented by a non-constant loop $\alpha:[x,y]\to X$ based at $x_0\in X$ and a scattered set $S_0\subset \alpha^{-1}(x_0)$ of $H$-cuts for $\alpha$. For such a choice of representative $\alpha$ and $S_0$, we inductively define scattered subspaces $S_{\kappa}\subseteq S_0$ for each ordinal $\kappa$ so that $S_{\kappa+1}$ consists of $x$, $y$, and all non-isolated points of $S_{\kappa}$. If $\kappa$ is a limit ordinal, then $S_{\kappa}=\bigcap_{\lambda<\kappa}S_{\lambda}$. Since $S_0$ is a countable set (recall Lemma \ref{scatteredordercharlemma}), there is a minimal countable ordinal $\kappa_{0}$ such that $S_{\kappa_{0}}=\{x,y\}$. We call $rk(\alpha,S_0)=\kappa_0$ the \textit{rank of} $\alpha$ \textit{with respect to} $S_0$ and note it's similarity to the Cantor-Bendixson rank of $S_0$ \cite{KechrisDST}. Observe that if $(a,b)\in \mathcal{I}(S_{\kappa+1})$, then $(a,b)$ is the union of the discrete (or empty) set $S_{\kappa}\cap (a,b)$ and the intervals in the set $\{(c,d)\in \mathcal{I}(S_{\kappa})\mid (c,d)\subseteq (a,b)\}$, which has order type $\mathbf{n}$, $\omega$, $\omega^{\ast}$, or $\zeta$.

\begin{lemma}\label{indlemma2}
Suppose $S_0$ is a set of $H$-cuts for $H$-scattered loop $\alpha:[x,y]\to X$ based at $x_0$. Suppose $(c,d)\in \mathcal{I}(S_{\kappa })$ and $(a,b)\subseteq (c,d)$ where $a,b\in S_{\lambda}$ for $\lambda\leq \kappa $. Then $rk(\alpha|_{[a,b]},S_0\cap[a,b])\leq \kappa $.
\end{lemma}

\begin{proof}
The sets $T_0=S_0\cap [a,b]$ and $U_0=S_0\cap [c,d]$ are sets of $H$-cuts for $\alpha|_{[a,b]}$ and $\alpha|_{[c,d]}$ respectively. Since $(c,d)\in \mathcal{I}(S_{\kappa })$, it is clear that $rk(\alpha|_{[c,d]},T_0)\leq \kappa $. For any $\lambda '$, every isolated point of $T_{\lambda '}\cap (a,b)= S_{\lambda '}\cap (a,b)$ is also an isolated point of $U_{\lambda '}\cap (c,d)= S_{\lambda '}\cap (c,d)$. Hence $rk(\alpha|_{[a,b]},T_0)\leq rk(\alpha|_{[c,d]},U_0)\leq \kappa $.
\end{proof}

\begin{corollary}\label{indlemma1}
Suppose $S_0$ is a set of $H$-cuts for $H$-scattered loop $\alpha:[x,y]\to X$. If $(c,d)\in \mathcal{I}(S_{\lambda})$, then $T_0=S_0\cap [c,d]$ is a set of $H$-cuts for $\alpha|_{[c,d]}$ and $rk(\alpha|_{[c,d]},T_0)\leq \lambda$.
\end{corollary}

\begin{lemma}\label{ordertypelemma}
If $H\leq \pionex$ and $\alpha:[x,y]\to X$ is a loop based at $x_0$ and $S\subseteq \alpha^{-1}(x_0)$ is a set of $\cl_{F,a_{\infty}}(H)$-cuts for $\alpha$ such that $\mathcal{I}(S)$ has order type $\mathbf{n}$, $\omega$, $\omega^{\ast}$, or $\zeta$, then $[\alpha]\in \cl_{F,a_{\infty}}(H)$.
\end{lemma}

\begin{proof}
If $\mathcal{I}(S)$ has order type $\mathbf{n}$, then $[\alpha]$ factors as a finite product of elements of the subgroup $\cl_{F,a_{\infty}}(H)$ and therefore lies in $\cl_{F,a_{\infty}}(H)$. If $\mathcal{I}(S)$ can be ordered as
\[(c_1,d_1)<(c_2,d_2)<(c_3,d_3)<\cdots \]
define $f:(\bbh,b_0)\to (X,x_0)$ so that $f\circ \ell_n\equiv \alpha|_{[c_n,d_n]}$. The continuity of $f$ follows from the continuity of $\alpha$. Since $f_{\#}(a_n)=[\alpha|_{[c,d]}]\in \cl_{F,a_{\infty}}(H)$ for all $n\in\bbn$, we have $[\alpha]=f_{\#}(a_{\infty})\in \cl_{F,a_{\infty}}(H)$. If $\mathcal{I}(S)$ has order type $\omega^{\ast}$, then $\mathcal{I}(\{x+y-t\mid t\in S\})$ has order type $\omega$. The previous case applies to give $[\alpha^{-}]\in \cl_{F,a_{\infty}}(H)$ and therefore $[\alpha]\in \cl_{F,a_{\infty}}(H)$. If $\mathcal{I}(S)$ has order type of $\zeta$, fix any $t\in (x,y)\cap S$. Then $\mathcal{I}(S\cap [t,y])$ has order type $\omega$ and $\mathcal{I}(S\cap [x,t])$ has order type $\omega^{\ast}$. It follows that $[\alpha|_{[t,y]}]\in \cl_{F,a_{\infty}}(H)$ and $[\alpha|_{[x,t]}]\in \cl_{F,a_{\infty}}(H)$ by the first and second cases respectively. Thus $[\alpha]=[\alpha|_{[x,t]}\cdot\alpha|_{[t,y]}]\in \cl_{F,a_{\infty}}(H)$.
\end{proof}

\begin{theorem}\label{mainthm}
For any subgroup $H\leq \pionex$, we have $\sct(H)\leq \cl_{F,a_{\infty}}(H)$.
\end{theorem}

\begin{proof}
Every non-trivial element of $\sct(H)$ is represented by a non-constant $H$-scattered loop $\alpha:[x,y]\to X$ based at $x_0$ that admits a set $S_0$ of $H$-cuts for $\alpha$. We show, by induction on $rk(\alpha,S_0)$, that $[\alpha]\in \cl_{F,a_{\infty}}(H)$ for all such pairs $(\alpha,S_0)$. If $rk(\alpha,S_0)=0$, then $S_0=\alpha^{-1}(x_0)=\{x,y\}$ and it follows that $[\alpha]\in H\leq \cl_{F,a_{\infty}}(H)$. Suppose that for any $H$-scattered loop $\beta:[x,y]\to X$ based at $x_0$ and choice of $H$-cuts set $S_0$ with $rk(\beta,S_0)<\kappa$, we have $[\beta]\in \cl_{F,a_{\infty}}(H)$. Consider $H$-scattered loop $\alpha$ based at $x_0$ with $H$-cut set $S_0$ such that $rk(\alpha,S_0)=\kappa$. 

If $\kappa=\lambda+1$ is a successor ordinal, then $\mathcal{I}(S_{\lambda})$ has order type $\mathbf{n}$, $\omega$, $\omega^{\ast}$, or $\zeta$. Fix $(c,d)\in \mathcal{I}(S_{\lambda})$ and $T=S_0\cap [c,d]$, then $T$ is a set of $H$-cuts for $\alpha|_{[c,d]}$ and, thus, $rk(\alpha|_{[c,d]},T)\leq \lambda$ by Corollary \ref{indlemma1}. Since $rk(\alpha|_{[c,d]},T)<\kappa$, we have $[\alpha_{[c,d]}]\in \cl_{F,a_{\infty}}(H)$ by our induction hypothesis. Since $[\alpha_{[c,d]}]\in \cl_{F,a_{\infty}}(H)$ for all $(c,d)\in \mathcal{I}(S_{\lambda})$, we have $[\alpha] \in \cl_{F,a_{\infty}}(H)$ by Lemma \ref{ordertypelemma}.

Recall that $\kappa$ is countable. So if $\kappa$ is a limit ordinal, we may find an $\omega$-sequence $0<\kappa_1<\kappa_2<\cdots <\kappa$ of countable ordinals converging to $\kappa$. Hence $\bigcap_{n\in\bbn}S_{\kappa_n}=\{x,y\}$. Replacing $\{\kappa_n\}$ with a subsequence if necessary, we may find $(a_n,b_n)\in \mathcal{I}(S_{\kappa_n})$ so that
\[(a_1,b_1)\subseteq (a_2,b_2)\subseteq \cdots \text{ and }\bigcup_{n\in\bbn}(a_n,b_n)=(x,y)\] 
Note that it is possible for at most one of the sequences $\{a_n\}$ or $\{b_n\}$ to be eventually constant. By Lemma \ref{indlemma1}, we have $rk(\alpha|_{[a_n,b_n]},S_0\cap [a_n,b_n])\leq \kappa_n<\kappa$. Therefore, $[\alpha|_{[a_n,b_n]}]\in \cl_{F,a_{\infty}}(H)$ for each $n\in \bbn$ by our induction hypothesis. 

If $a_{n+1}<a_n<b_{n+1}$, then by Lemma \ref{indlemma2}, we have $rk(\alpha|_{[a_{n+1},a_n]},S_0\cap [a_{n+1},a_n])\leq \kappa_{n+1}<\kappa$ and, thus, $[\alpha|_{[a_{n+1},a_n]}]\in \cl_{F,a_{\infty}}(H)$. Similarly, if $a_{n+1}<b_n<b_{n+1}$, then $[\alpha|_{[b_{n},b_{n+1}]}]\in \cl_{F,a_{\infty}}(H)$. Now $A=\{x,y\}\cup\{a_n\mid n\in\bbn\}\cup \{b_n\mid n\in\bbn\}$ is a closed set in $A$ such that $\mathcal{I}(A)$ has order type $\omega$, $\omega^{\ast}$, or $\zeta$ and if $(c,d)\in \mathcal{I}(A)$, then $[\alpha|_{[c,d]}]\in  \cl_{F,a_{\infty}}(H)$. By Lemma \ref{ordertypelemma}, we have $[\alpha]\in \cl_{F,a_{\infty}}(H)$. This completes the induction.
\end{proof}

\begin{corollary}\label{equality1cor}
For any $H\leq \pionex$, we have\[\cl_{F,a_{\infty}}(H)=\sct(\cl_{F,a_{\infty}}(H))=\cl_{F,a_{\infty}}(\sct(H)).\]
\end{corollary}

\begin{proof}
Applying the operator properties of $\cl_{F,a_{\infty}}$ (Lemma \ref{closurepropertieslemma}) and $\sct$ (Lemma \ref{sctproperties}) and Theorem \ref{mainthm}, the following sequence of inclusions proves the desired equalities.
\begin{eqnarray*}
\cl_{F,a_{\infty}}(H) &\leq & \sct(\cl_{F,a_{\infty}}(H))\\
&\leq & \sct(\cl_{F,a_{\infty}}(\sct(H)))\\
&\leq & \cl_{F,a_{\infty}}(\cl_{F,a_{\infty}}(\sct(H)))\\
&=& \cl_{F,a_{\infty}}(\sct(H))\\
&\leq & \cl_{F,a_{\infty}}(\cl_{F,a_{\infty}}(H))\\
&=& \cl_{F,a_{\infty}}(H)
\end{eqnarray*}
\end{proof}

\begin{corollary}
For any subgroup $H\leq \pionex$, the following are equivalent:
\begin{enumerate}
\item $\sct(H)$ is $(F,a_{\infty})$-closed,
\item $\sct(H)=\cl_{F,a_{\infty}}(H)$,
\item $\sct(H)=\sct(\sct(H))$.
\end{enumerate}
\end{corollary}

\begin{proof}
By Theorem \ref{mainthm} and Corollary \ref{equality1cor}, the following inequality holds for all $H$:
\[\sct(H)\leq \sct(\sct(H))\leq \cl_{F,a_{\infty}}(\sct(H))=\cl_{F,a_{\infty}}(H).\]
(1) $\Rightarrow$ (2) $\Rightarrow$ (3) follows immediately. Suppose (3) holds, and $f:(\bbh,b_0)\to (X,x_0)$ is a map such that $f_{\#}(F)\leq \sct(H)$. Then $\{1\}\cup \{\frac{n-1}{n}\mid n\in\bbn\}$ is a closed scattered set of $\sct(H)$-cuts for $f\circ \ell_{\infty}$. Hence $f_{\#}(a_{\infty})=[f\circ\ell_{\infty}]\in\sct(\sct(H))=\sct(H)$. This proves $\sct(H)$ is $(F,a_{\infty})$-closed. 
\end{proof}

\section{The Homotopically Hausdorff Property and Scattered Products in fundamental groups}

\begin{definition}{\cite{CMRZZ08,FZ07}}\label{homhausdorffdef}
We say a path-connected space $X$ is \textit{homotopically Hausdorff relative to a subgroup} $H\leq \pionex$ if for every $x\in X$ and every path $\alpha:\ui\to X$ from $\alpha(0)=x_0$ to $\alpha(1)=x$, only the trivial right coset of $H^{\alpha}=[\alpha^{-}]H[\alpha]$ in $\pi_1(X,x)$ has arbitrarily small representatives, that is, if for every $g\in \pi_1(X,x)\backslash H^{\alpha}$, there is an open neighborhood $U$ of $x$ such that there is no loop $\delta:(\ui,\{0,1\})\to (U,x)$ with $H^{\alpha}g=H^{\alpha}[\delta]$.

If $X$ is homotopically Hausdorff relative to the trivial subgroup $H=1$, we say $X$ is \textit{homotopically Hausdorff}.
\end{definition}

\begin{remark}
A space $X$ is homotopically Hausdorff if and only if for every point $x\in X$, there are no non-trivial elements of $\pi_1(X,x)$ that have a representative loop in every neighborhood of $x$.
\end{remark}

Let $\bbhp=\bbh\cup ([-1,0]\times \{0\})$ be the Hawaiian earring with a ``whisker" attached with basepoint $\bpp=(-1,0)$ and where $\iota:\ui\to \bbhp$, $\iota(t)=(t-1,0)$ is the inclusion of the whisker. Define $c_n=[\iota]a_{n}[\iota^{-}]$, $c_{\infty}=[\iota]a_{\infty}[\iota^{-}]$, and $C=[\iota]F[\iota^{-}]$ in $\pi_1(\bbhp,\bpp)$. Equipped with Theorem \ref{mainthm}, we provide a thorough study of the $(C,c_{\infty})$-closure and explicitly compute the closures needed for our characterizations of the homotopically Hausdorff property.


\begin{proposition}\label{inclusionprop}
For any $H\leq\pionex$, we have $\cl_{F,a_{\infty}}(H)\leq \cl_{C,c_{\infty}}(H)$.
\end{proposition}

\begin{proof}
The retraction $f:\bbhp\to\bbh$ collapsing the whisker to a point satisfies $f_{\#}(C)=F$ and $f_{\#}(c_{\infty})=a_{\infty}$. Hence, $a_{\infty}\in \cl_{C,c_{\infty}}(F)$ and we may apply Remark \ref{comparisonremark}.
\end{proof}

\begin{corollary}\label{equality2cor}
For any $H\leq \pionex$, we have \[\cl_{C,c_{\infty}}(H)=\sct(\cl_{C,c_{\infty}}(H))=\cl_{C,c_{\infty}}(\sct(H)).\]
\end{corollary}

\begin{proof}
For any $H\leq \pionex$, Theorem \ref{mainthm} and Proposition \ref{inclusionprop} give the inclusion $\sct(H)\leq \cl_{C,c_{\infty}}(H)$. Hence, the analogous sequence of inclusions used in the proof of Corollary \ref{equality1cor} applies.
\end{proof}

\begin{corollary}\label{ccinfcorollary}
For any $H\leq \pionex$, we have $\sct(H)=\cl_{C,c_{\infty}}(H)$ if and only if $\sct(H)$ is $(C,c_{\infty})$-closed.
\end{corollary}

\begin{definition}
For a space $X$ and subset $A\subseteq X$, let 
\begin{itemize}
\item[] $F(X,A)=\{[\alpha]\in \pionex\mid \alpha^{-1}(A)\text{ is finite or }\alpha\text{ is constant}\}$,
\item[] $\sct(X,A)=\{[\alpha]\in \pionex\mid \alpha^{-1}(A)\text{ is a scattered space or }\alpha\text{ is constant}\}$.
\end{itemize}
\end{definition}

Note that $F(X,A)$ and $\sct(X,A)$ are subgroups of $\pionex$ such that $\sct(X,A)\leq \sct(F(X,A))$. In fact, since the closed scattered subspaces of $\bbr$ are precisely the closed countable subsets of $\bbr$, $\sct(X,A)$ is precisely the subgroup of countable cut-points $CCP(X,A,x_0)$ in \cite[Example 3.13]{BFTestMap}.

\begin{example}
As a special case, we call the group $Sc=\sct(\bbh,\{b_0\})\leq \pioneh$ the \textit{subgroup of scattered words}. It is known that $Sc$ is algebraically free on uncountably many generators (See \cite{CChegroup} or \cite{EK99subgroups}).
\end{example}

\begin{proposition}\cite[Proposition 3.14]{BFTestMap}\label{sctxaisccinfclosedprop}
If $X$ is a one-dimensional metric space and $A\subseteq X$ is closed, then $\sct(X,A)$ is $(C,c_{\infty})$-closed.
\end{proposition}

\begin{lemma}\label{sctequalitylemma}
If $X$ is a one-dimensional metric space and $A\subseteq X$ is closed, then $\sct(F(X,A))=\sct(X,A)$.
\end{lemma}

\begin{proof}
The inclusion $\sct(X,A)\leq \sct(F(X,A))$ is clear. Suppose $1\neq [\alpha]\in \sct(F(X,A))$ for $F(X,A)$-scattered loop $\alpha$. Take $K\subseteq \alpha^{-1}(x_0)$ to be a set of $F(X,A)$-cuts for loop $\alpha:\ui\to X$. We define a new path $\beta:\ui\to X$ as follows: fix $(c,d)\in \mathcal{I}(K)$. Then $\alpha|_{[c,d]}$ is path-homotopic to a path $\gamma:[c,d]\to X$ such that $\gamma^{-1}(A)$ is finite. Since reduction of any path only occurs within the image of that path, the preimage of $A$ under the reduced representative of $\gamma$ is also finite. However, this reduced representative is unique to the homotopy class and is therefore also the reduced representative of $\alpha|_{[c,d]}$. Hence, we define $\beta$ so that $\beta|_K=\alpha|_{K}$ and $\beta|_{[c,d]}$ is the reduced representative of $\alpha|_{[c,d]}$ for all $(c,d)\in\mathcal{I}(K)$. Continuity of $\beta$ at a point $k\in K$ follows from the continuity of $\alpha$ at $k$ and the fact that $\beta([c,d])\subset \alpha([c,d])$ for all $(c,d)\in\mathcal{I}(K)$. Since, to construct $\beta$, we have only performed reductions of subpaths of $\alpha$, it is clear that $\alpha$ and $\beta$ have the same reduced representative and are therefore homotopic. Notice that $\beta^{-1}(A)$ is closed and is the union of $K$ and at most finitely many isolated points in each interval $(c,d)\in\mathcal{I}(K)$. Hence, since $K$ is scattered, $\beta^{-1}(A)$ is scattered. Therefore, $[\alpha]=[\beta]\in \sct(X,A)$.
\end{proof}

%
\begin{theorem}\label{onedtheorem}
If $X$ is a one-dimensional metric space and $A\subseteq X$ is closed, then $\cl_{C,c_{\infty}}(F(X,A))=\sct(X,A)$.
\end{theorem}

\begin{proof}
Together, Proposition \ref{sctxaisccinfclosedprop} and Lemma \ref{sctequalitylemma} give that $\sct(X,A)=\sct(F(X,A))$ is $(C,c_{\infty})$-closed. Applying Corollary \ref{ccinfcorollary} to $H=F(X,A)$ gives $\sct(F(X,A))=\cl_{C,c_{\infty}}(F(X,A))$. Hence, $\cl_{C,c_{\infty}}(F(X,A))=\sct(X,A)$.
\end{proof}

\begin{corollary}\label{agree}
$\cl_{C,c_{\infty}}(F)=\sct(F)=Sc$ and $\cl_{C,c_{\infty}}(C)=[\iota]Sc[\iota^{-}]$.
\end{corollary}

\begin{proof}
Observe that $F=F(\bbh,\{b_0\})$, $Sc=\sct(\bbh,\{b_0\})$, $C=F(\bbhp,\{b_0\})\leq \pi_1(\bbhp,\bpp)$, and $\sct(\bbhp,\{b_0\})=[\iota]\sct(\bbh,\{b_0\})[\iota^{-}]=[\iota]Sc[\iota^{-}]$. Applying Theorem \ref{onedtheorem} gives both equalities.
\end{proof}

We complete this section by proving the equivalence (1) $\Leftrightarrow$ (2) $\Leftrightarrow$ (5) in Theorem \ref{overallthm}. Let $\ell_{\geq m}$ denote the infinite concatenation $\prod_{n=m}^{\infty}\ell_n$ for $m\in\bbn$.

\begin{theorem}\label{groupthm}
For any space $X$, the following are equivalent:
\begin{enumerate}
\item the trivial subgroup of $\pionex$ is $(C,c_{\infty})$-closed,
\item for any maps $f,g:(\bbh,b_0)\to (X,x)$ such that $f_{\#}(a_n)=g_{\#}(a_n)$ for all $n\in\bbn$, we have $f_{\#}(a_{\infty})=g_{\#}(a_{\infty})$,
\item for any maps $f,g:(\bbh,b_0)\to (X,x)$ such that $f_{\#}(a_n)=g_{\#}(a_n)$ for all $n\in\bbn$, we have $f_{\#}|_{Sc}=g_{\#}|_{Sc}$, i.e. $f_{\#}$ and $g_{\#}$ agree on the subgroup of scattered words.
\end{enumerate}
If $X$ is first countable, these are also equivalent to the homotopically Hausdorff property.
\end{theorem}

\begin{proof}
(1) $\Rightarrow$ (2) Suppose $1\leq \pionex$ is $(C,c_{\infty})$-closed and that $f,g:(\bbh,b_0)\to (X,x)$ are maps satisfying $f_{\#}(a_n)=g_{\#}(a_n)$ for all $n\in\bbn$. Define $\gamma_n=(f\circ\ell_{\geq n})^{-}\cdot (g\circ\ell_{\geq n})$, $n\in\bbn$. For all $n$, we have $[(f\circ\ell_n)^{-}\cdot (g\circ \ell_n)]=1$ and, thus, $\gamma_n\simeq \gamma_{n+1}$. Let $\alpha:\ui\to X$ be any path from $x_0$ to $x$ and define $k:(\bbhp,\bpp)\to (X,x_0)$ by $k\circ\iota=\alpha$ and $k\circ\ell_n=\gamma_{n}\cdot\gamma_{n+1}^{-}$. Then $k_{\#}(C)=1$ and a swindle yields $k_{\#}(c_{\infty})=[\alpha][\gamma_1][\alpha^{-}]$. Since the trivial subgroup is assumed to be $(C,c_{\infty})$-closed, we must have $[\gamma_1]=f_{\#}(a_{\infty})^{-1}g_{\#}(a_{\infty})=1$ in $\pi_1(X,x)$. Thus $f_{\#}(a_{\infty})=g_{\#}(a_{\infty})$.

(2) $\Rightarrow$ (3) Suppose (2) holds and $f,g:(\bbh,b_0)\to (X,x)$ are maps satisfying $f_{\#}|_{F}=g_{\#}|_{F}$. By Corollary \ref{agree}, we have $Sc=\cl_{C,c_{\infty}}(F)$ so we show that $f_{\#}$ and $g_{\#}$ agree on every element of $\cl_{C,c_{\infty}}(F)$. We proceed by induction on the $(F,a_{\infty})$-rank of the elements of $\cl_{F,a_{\infty}}(F)$. By assumption, $f_{\#}$ and $g_{\#}$ agree on the elements of $\cl_{F,a_{\infty}}(F)_{\mathbf{0}}=F$. Suppose $f_{\#}(a)=g_{\#}(a)$ for all elements $a\in\cl_{F,a_{\infty}}(F)$ with $(F,a_{\infty})$-rank $<\kappa$. Suppose $a\in \cl_{F,a_{\infty}}(F)$ has $(F,a_{\infty})$-rank $\kappa$. Since $\kappa$ must necessarily be a successor ordinal, we have $\kappa=\lambda+1$ for some ordinal $\lambda$. Then $a=\prod_{i=1}^{m}b_i$ where either $b_i$ has $(F,a_{\infty})$-rank $\leq \lambda$ or $b_i=k_{\#}(a_{\infty})$ for $k:(\bbh,b_0)\to(\bbh,b_0)$ with $k_{\#}(F)\leq \cl_{F,a_{\infty}}(F)_{\lambda}$. In the first case, $f_{\#}$ and $g_{\#}$ agree on $b_i$ by the induction hypothesis so we focus on the case where $b_i=k_{\#}(a_{\infty})$. For each $n\in\bbn$, we have $k_{\#}(a_n)\in \cl_{F,a_{\infty}}(F)_{\lambda}$ and, thus, $f_{\#}(k_{\#}(a_n))=g_{\#}(k_{\#}(a_n))$ by the induction hypothesis. Since (2) holds and $f\circ k$, $g\circ k$ are maps $\bbh\to\bbh$ whose induced homomorphisms agree on $F$, we have $f_{\#}(b_i)=f_{\#}(k_{\#}(a_{\infty}))=g_{\#}(k_{\#}(a_{\infty}))=g_{\#}(b_i)$. Since $f_{\#}$ and $g_{\#}$ agree on the factors of $a$, we have $f_{\#}(a)=g_{\#}(a)$, completing the induction.

(3) $\Rightarrow$ (1) Suppose $1\leq \pionex$ is not $(C,c_{\infty})$-closed. Then there exists, $k:(\bbhp,\bpp)\to (X,x_0)$ such that $k_{\#}(c_n)=1$ for all $n\in\bbn$ and $k_{\#}(c_{\infty})\neq 1$. Let $x=k(b_0)$, $f=k|_{\bbh}$, and $g:\bbh\to X$ be the constant map at $x$. Then $f_{\#}(a_n)=1=g_{\#}(a_n)$ in $ \pi_1(X,x)$ for all $n$. Using the basepoint change isomorphism $\phi:\pi_1(X,x_0)\to \pi_1(X,x)$, $\phi([\alpha])=[(k\circ\iota)^{-}\cdot \alpha\cdot(k\circ\iota)]$, we see that $f_{\#}(a_{\infty})=\phi(k_{\#}(c_{\infty}))\neq 1$. Hence, $f_{\#}$ and $g_{\#}$ do not agree on $a_{\infty}\in \cl_{F,a_{\infty}}(F)=Sc$.

The last statement of the theorem is a special case of Theorem 3.6 in \cite{BFTestMap}. We include a proof for the sake of completeness.

(hom. Hausdorff) $\Rightarrow$ (1) If $1\leq \pionex$ is not $(C,c_{\infty})$-closed, then there exists a map $f:(\bbhp,\bpp)\to (X,x_0)$ such that $f_{\#}(C)=1$ and $f_{\#}(c_{\infty})\neq 1$. Setting $\alpha=f\circ\iota$, we have $f_{\#}(a_n)=[\alpha^{-}]f_{\#}(c_{n})[\alpha]=1$ for all $n\in\bbn$ and $g=f_{\#}(a_{\infty})\neq 1$. Let $U$ be any open neighborhood of $x=f(b_0)$. By the conintuity of $f$, there exists an $m\in\bbn$ such that $\delta=f\circ\ell_{\geq m}$ has image in $U$. Then $g=f_{\#}(a_1a_2\dots a_{m-1})[\delta]=[\delta]$, showing that $X$ is not homotopically Hausdorff.

(1) and $X$ first countable $\Rightarrow$ (hom. Hausdorff) Suppose $X$ is first countable and is not homotopically Hausdorff. Then there exists $1\neq g\in\pionex$ and a path $\alpha$ starting at $x_0$ such that for every neighborhood $U$ of $x=\alpha(1)$, there exists a loop $\delta$ in $U$ based at $x$ such that $g=[\delta]$. Let $U_1\supseteq U_2\supseteq U_3\supseteq\cdots$ be a neighborhood base at $x$. By assumption, there exists loops $\delta_n$ in $U_n$ based at $x$ such that $g=[\delta_n]$. Define $f:\bbhp\to X$ by $f\circ\iota=\alpha$ and $f\circ\ell_n= \delta_n\cdot\delta_{n+1}^{-}$. Since the basis $\{U_n\mid n\in\bbn\}$ is nested, $f$ is continuous at $b_0$. Note that $f_{\#}(c_n)=[\alpha\cdot \delta_{n}\cdot\alpha^{-}][\alpha\cdot \delta_{n+1}\cdot\alpha^{-}]^{-1}=1$ for all $n\in\bbn$ and, thus, $f_{\#}(C)=1$. However, applying a swindle gives $f_{\#}(c_{\infty})=[\alpha\cdot \delta_{1}]\left[\prod_{n= 2}^{\infty}(\delta_{n}^{-}\cdot\delta_{n})\right][\alpha^{-}]=[\alpha\cdot \delta_{1}\cdot\alpha^{-}]\neq 1$. We conclude that the trivial subgroup $1\leq \pionex$ is not $(C,c_{\infty})$-closed.
\end{proof}

\begin{example}\label{archipelagoexample}
Let $\mathbb{HA}\subseteq \mathbb{R}^3$ be the harmonic archipelago space \cite{CHM} where $\bbr^2\times \{0\}\cap \mathbb{HA}=\mathbb{H}\times \{0\}$. Then infinite $\pi_1$-products are not well-defined in $\pi_1(\mathbb{HA},b_0)$. Note that $1\neq [\ell_1]=[\ell_2]=\cdots $ in $\pi_1(\mathbb{HA},b_0)$ and define $\gamma_n=\ell_{n}\cdot \ell_{n+1}^{-}$. If $f:\bbh\to \mathbb{HA}$ is the constant map at $b_0$ and $g:\bbh\to \mathbb{HA}$ satisfies $g\circ \ell_n=\ell_{n}\cdot\ell_{n+1}^{-}$, then we have $f_{\#}(a_n)=g_{\#}(a_n)$ for all $n\in\bbn$. However, $f_{\#}(a_{\infty})=1\neq [\ell_1]= g_{\#}(a_{\infty})$. Hence, $\prod_{n=1}^{\infty}[\gamma_n]$ does not have a well-defined meaning in $\pi_1(\mathbb{HA},b_0)$.
\end{example}

In general, if a space fails to have the homotopically Hausdorff property, one is guaranteed the existence of a net of loops based at a point $x$, within a single non-trivial homotopy class, that converges to the constant loop at $x$ in the compact-open topology. However, to form topologically represented infinite products in $\pi_1$, one must have a \textit{sequence} of such loops. Without assuming first countability, it may not be possible to replace a net with a cofinal subsequence. The following higher-ordinal analogue of the harmonic archipelago illustrates why it is necessary to assume first countability if we wish to include the homotopically Hausdorff property in the statements of Theorem \ref{overallthm} and Theorem \ref{groupthm}. 

\begin{example}\label{omegaonesuspension}
Let $X$ be the reduced suspension of the first uncountable compact ordinal $\omega_1+1=\omega_1\cup\{\omega_1\}$ with canonical basepoint $x_0$. For $\lambda<\omega_1$, let $C_{\lambda}$ be the circle, which is the image of $\{\lambda\}\times\ui$ in $X$, and let $\gamma_{\lambda}$ be the canonical loop traversing $C_{\lambda}$. Since $C_{\lambda}$ is a retract of $X$, each $\gamma_{\lambda}$ is homotopically non-trivial. Let $Y$ be the space obtained by attaching a $2$-cell to $X$ along the loop $\gamma_{\mu}\cdot\gamma_{\lambda}^{-}$ for all ordered pairs $\mu<\lambda<\omega_1$. Since no sequence of countable ordinals converges to $\omega_1$, there is no sequentially compact subspace of $X$ that contains $C_{\lambda}$ for infinitely many $\lambda$. Therefore, infinite concatenations of sequences of non-trivial loops can be formed in neither $X$ nor $Y$. It follows that $\pi_1(X,x_0)$ is free on the uncountable set $\{[\gamma_{\lambda}]\mid \lambda<\omega_1\}$, the group $\pi_1(Y,y_0)$ is infinite cyclic (generated by the homotopy class containing $\{\gamma_{\lambda}\mid \lambda<\omega_1\}$), and the trivial subgroup is $(C,c_{\infty})$-closed in both $\pionex$ and $\pi_1(Y,x_0)$. However, since every neighborhood of $x_0$ in $Y$ contains $C_{\lambda}$ for some $\lambda<\omega_1$, $Y$ is not homotopically Hausdorff.
\end{example}

\section{Scattered products in fundamental groupoids}

We now introduce a set of closure pairs, which will be useful for transitioning from infinite products in fundamental groups to infinite products in fundamental groupoids.

Let $A$ be a closed proper subset of $\ui$ containing $\{0,1\}$. For each $I=(a,b)\in\mci(A)$, let $C_I=\left\{(x,y)\in\bbr^2\mid y\geq 0,\left(x-\frac{a+b}{2}\right)^2+y^2=\left(\frac{b-a}{2}\right)^2\right\}$ be the semicircle whose boundary is $\{(a,0),(b,0)\}$. Let $\bbw_{A}=B\cup \bigcup_{I\in\mci(A)}C_I$ with basepoint $w_0=(0,0)$ where $B=[0,1]\times\{0\}$ is the \textit{base-arc}.

\begin{figure}[H]
\centering \includegraphics[height=1.1in]{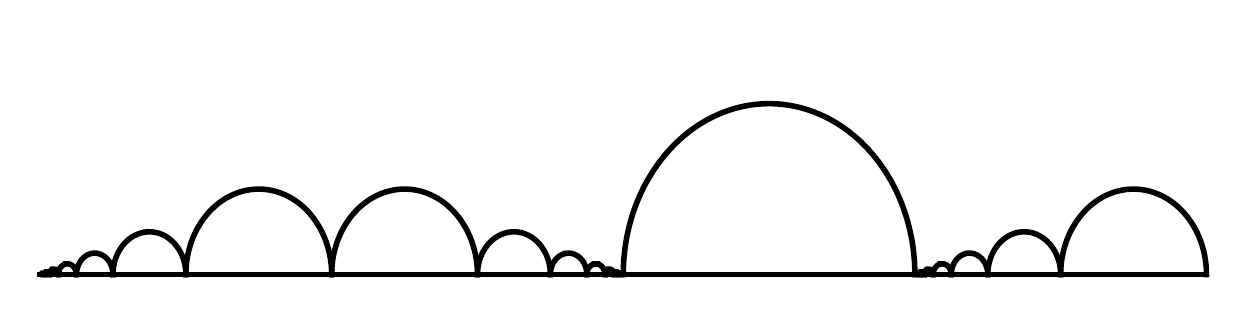}
\caption{\label{waspace}An example of the space $\bbw_A$ where $A$ is scattered and $\mci(A)$ has order type $\zeta+\mathbf{1}+\omega^{\ast}$.}
\end{figure}

For $I=(a,b)\in\mci(A)$, let $\lambda_I:\ui\to B$ be the path $\lambda_I(t)=(bt+a(1-t),0)$ and $\upsilon_{I}:\ui\to C_I$ be the path so that if $r:\bbw_A\to B$ is the projection onto the x-axis, then $r\circ\upsilon_{I}=\lambda_I$. Let $\lambda(t)=(t,0)$ and $\upsilon:\ui\to (A\times \{0\})\cup\bigcup_{I\in\mci(A)}C_I$ be the path such that $r\circ\upsilon=\lambda$. Hence, $\lambda$ and $\upsilon$ are paths in $\bbw_A$ from $w_0$ to $(1,0)$ such that $\lambda$ is the path along the base-arc and $\upsilon$ is the path through the circles $C_I$.

Let $W_A$ be the subgroup of $\pi_1(\bbw_{A},w_0)$ generated by the elements $w_I=[\upsilon_{[0,b]}\cdot \lambda_{[a,b]}^{-}\cdot \upsilon_{[0,a]}^{-}]$ for $I=(a,b)\in\mci(A)$ and let $w_{A,\infty}=[\upsilon\cdot\lambda^{-}]$. We consider the closure pair $(W_A,w_{A,\infty})$ for the test space $(\bbw_{A},w_0)$.

\begin{remark}
If $A$ is finite, then $\cl_{W_A,w_{A,\infty}}$ is the identity closure operator. If $\mcc$ is the ternary Cantor set, then the operator $\cl_{W_{\mcc},w_{\mcc,\infty}}$ is the closure pair used in \cite[Section 7]{BFTestMap} to characterize the well-definedness of transfinite $\Pi_1$-products.
\end{remark}

\begin{lemma}\label{omegapathlemma}
Let $A=\{\frac{n-1}{n}\mid n\in\bbn\}\cup\{1\}$. Then for any subgroup $H\leq \pionex$, $\cl_{C,c_{\infty}}(H)=\cl_{W_A,w_{A,\infty}}(H)$.
\end{lemma}

\begin{proof}
By Remark \ref{comparisonremark}, it suffices to show $w_{A,\infty}\in \cl_{C,c_{\infty}}(W_A)$ and $c_{\infty}\in \cl_{W_A,w_{A,\infty}}(C)$. Let $t_n=\frac{n-1}{n}$ so that the components of $\ui\backslash A$ are $I_n=\left(t_n,t_{n+1}\right)$, $n\in\bbn$. Set $\upsilon_n:=\upsilon_{I_n}$, $\lambda_n:=\lambda_{I_n}$, and $w_n:=w_{I_n}$. Let $\upsilon_{\geq n}=\prod_{k=n}^{\infty}\upsilon_k$ and $\lambda_{\geq n}(t)=\lambda\left(t+(1-t)t_n\right)$ be the path along the base-arc from $\left(t_n,0\right)$ to $(1,0)$.

Define $f:(\bbhp,\bpp)\to (\bbw_A,w_0)$ so that $f\circ\iota=\upsilon$ and $f\circ \ell_{n}\equiv \upsilon_{\geq n}^{-}\cdot \lambda_n\cdot\upsilon_{\geq n+1}$. Then $f_{\#}(c_n)=[\upsilon\cdot\upsilon_{\geq n}^{-}\cdot \lambda_n\cdot\upsilon_{\geq n+1}\cdot \upsilon^{-}]=[\upsilon_{[0,t_n]}\cdot \lambda_n\cdot \upsilon_{[0,t_{n+1}]}^{-}]=(w_n)^{-1}\in W_A$ and $f_{\#}(c_{\infty})=\left[\upsilon\cdot\left( \prod_{n=1}^{\infty}\upsilon_{\geq n}^{-}\cdot \lambda_n\cdot\upsilon_{\geq n+1}\right)\cdot \upsilon^{-}\right]=[\lambda\cdot\upsilon^{-}]=w_{A,\infty}^{-1}$. This implies $w_{A,\infty}^{-1}\in \cl_{C,c_{\infty}}(W_A)$, hence $w_{A,\infty}\in\cl_{C,c_{\infty}}(W_A)$.

Define $g:(\bbw_A,w_0)\to (\bbhp,\bpp)$ so that $g\circ\upsilon_1\equiv g\circ\lambda_1\equiv \iota$ and $g\circ \upsilon_{n}$ is constant at $b_0$ for $n\geq 2$ and $g\circ \lambda_n=\ell_{n-1}$ for $n\geq 2$. Then $g_{\#}(w_1)=1$, $g_{\#}(w_n)=[\iota \cdot (g\circ\lambda_{n}^{-})\cdot \iota^{-}]=c_{n-1}^{-1}\in C$ for all $n\geq 2$, and $g_{\#}(w_{A,\infty})=[\iota \cdot (g\circ \lambda_{n\geq 2}^{-})\cdot \iota^{-}]=c_{\infty}^{-1}$. This implies $c_{\infty}^{-1}\in \cl_{W_A,w_{A,\infty}}(C)$, hence $c_{\infty}\in \cl_{W_A,w_{A,\infty}}(C)$.
\end{proof}

\begin{definition}\label{relscatteredpathproddef}
Let $H\leq \pionex$ be a subgroup. The \textit{scattered path-product closure of} $H$, denoted $\cl_{spp}(H)$, is the subgroup of $\pionex$ generated by the union $\bigcup_{A}\cl_{W_A,w_{A,\infty}}(H)$ taken over all closed scattered sets $A\subset\ui$ containing $\{0,1\}$. We say that $X$ has \textit{well-defined scattered $\Pi_1$-products relative to} $H$ if $H=\cl_{spp}(H)$.
\end{definition}

When $H=1$, the condition $\cl_{spp}(1)=1$ in $\pionex$ is clearly equivalent to well-definedness of scattered $\Pi_1$-products in $X$ (recall (4) of Definition \ref{maindef}). The following lemma is an immediate consequence of the definition of $\cl_{spp}$ and the closure properties of $\cl_{W_A,w_{A,\infty}}$ (recall Lemma \ref{closurepropertieslemma}).

\begin{lemma}
$\cl_{spp}(H)=H$ if and only if $H$ is $(W_A,w_{A,\infty})$-closed for every closed scattered set $A\subset\ui$ containing $\{0,1\}$. In particular, the operator $\cl_{spp}$ has all of the properties appearing in Lemma \ref{closurepropertieslemma}.
\end{lemma}

\begin{remark}
If $N\trianglelefteq \pionex$ is a normal subgroup, then $\cl_{W_A,w_{A,\infty}}(N)$ and $\cl_{W_{A'},w_{A',\infty}}(N)$ are contained in the normal subgroup $\cl_{W_B,w_{B,\infty}}(N)$ where $B=\{0\}\cup \{\frac{t+1}{3}\mid t\in A\}\cup \{\frac{t+2}{3}\mid t\in A'\}$ (See \cite[Lemma 2.9]{BFTestMap}). Hence, $\cl_{spp}(N)=\bigcup_{A}\cl_{W_A,w_{A,\infty}}(N)$ is a normal subgroup. 
\end{remark}

\begin{theorem}\label{mainspathproducts}
A space $X$ has well-defined scattered $\Pi_1$-products relative to $H$ if and only if $H$ is $(C,c_{\infty})$-closed. Moreover, if $X$ is first countable, these properties are equivalent to the homotopically Hausdorff relative to $H$ property.
\end{theorem}

\begin{proof}
We focus on the first statement since \cite[Theorem 3.6]{BFTestMap} states that homotopically Hausdorff rel. $H$ $\Rightarrow$ $H$ is $(C,c_{\infty})$-closed and that the converse holds when $X$ is first countable.

Suppose $X$ has well-defined scattered $\Pi_1$-products rel. $H\leq\pionex$, i.e. $\cl_{spp}(H)=H$. Then $\cl_{W_A,w_{\infty}}(H)=H$ for $A= \{\frac{n-1}{n}\mid n\in\bbn\}\cup\{1\}$. By Lemma \ref{omegapathlemma}, we have $\cl_{C,c_{\infty}}(H)=\cl_{W_A,w_{\infty}}(H)=H$, showing that $H$ is $(C,c_{\infty})$-closed. For the converse, suppose $H$ is $(C,c_{\infty})$-closed. Fix a closed scattered set $A\subseteq\ui$ containing $\{0,1\}$. We check that $H$ is $(W_A,w_{A,\infty})$-closed. Suppose $f:(\bbw_A,w_0)\to (X,x_0)$ is a map such that $f_{\#}(W_A)\leq H$. We will prove that $f_{\#}(w_{A,\infty})=f_{\#}([\upsilon\cdot \lambda^{-}])\in H$.

Identify $A$ with the subspace $A\times \{0\}$ of $\bbw_A$. Since $F(\bbw_A,A)=1$ if $0$ is a limit point of $A$, we must consider an alternative test space. Let $\mathbb{Y}$ be the space obtained by attaching a 1-cell $L_I$ to $\bbw_A$ with boundary $\{w_0,a\}$ for each $I=(a,b)\in\mci(A)$ with $a>0$. We give $\mathbb{Y}$ the weak topology with respect to the subspaces $\bbw_A$ and $L_I$ for $I=(a,b)\in\mci(A)$ with $a>0$. Although $\mathbb{Y}$ will not always be metrizable, $\mathbb{Y}$ is homotopy equivalent to a one-dimensional subspace $\mathbb{Y}'\subset\bbr^3$ by a bijective homotopy equivalence $\mathbb{Y}\to\mathbb{Y}'$ which maps $\bbw_A$ to $\bbw_A\times \{0\}$ and each 1-cell $L_I$ to an arc of diameter $\geq 1$ in $\bbr^3$ that connects $(0,0,0)$ to $(a,0,0)$ for each $I=(a,b)\in\mci(A)$ with $a>0$. Therefore, we may apply results for one-dimensional metric spaces to $\mathbb{Y}$.

Let $i:\bbw_A\to \mathbb{Y}$ be the inclusion map and $\rho_I:\ui\to \mathbb{Y}$ be the arc from $w_0$ to $a=\inf(I)$ along the attached 1-cell $L_I$. Define $g:\mathbb{Y}\to X$ so that $g|_{\bbw_A}=f$ and $g\circ \rho_I\equiv \upsilon|_{[0,a]}$ for $I=(a,b)\in\mathcal{I}(A)$, $a\neq 0$. The continuity of $g$ is clear since $\mathbb{Y}$ has the weak topology. Notice that $[\upsilon\cdot\lambda^{-}]\in \sct(\mathbb{Y},A)$.

First, observe that $g_{\#}(F(\mathbb{Y},A)) \leq  H$. Since a detailed proof of this inclusion requires a straightforward but somewhat tedious decomposition of the elements of $F(\mathbb{Y},A)$, we give a separate proof below in Lemma \ref{technicallemma}. Recalling the existence of a bijective homotopy equivalence between $\mathbb{Y}$ and a one-dimensional metric space, we have $\sct(\mathbb{Y},A)=\cl_{C,c_{\infty}}(F(\mathbb{Y},A))$ by Theorem \ref{onedtheorem}. Combining this with the inclusion $g_{\#}(F(\mathbb{Y},A))\leq H$ and the closure properties from Lemma \ref{closurepropertieslemma}, we have
\begin{eqnarray*}
f_{\#}(w_{A,\infty})=g_{\#}([\upsilon\cdot\lambda^{-}]) &\in & g_{\#}(\sct(\mathbb{Y},A))\\
&= & g_{\#}(\cl_{C,c_{\infty}}(F(\mathbb{Y},A)))\\
& \leq & \cl_{C,c_{\infty}}(g_{\#}(F(\mathbb{Y},A)))\\
& \leq & \cl_{C,c_{\infty}}(H)\\
&=& H
\end{eqnarray*}
where the last equality holds since $H$ is $(C,c_{\infty})$-closed.
\end{proof}


\begin{lemma}
\label{technicallemma}
Consider the spaces, maps, and groups as constructed in the proof of Theorem \ref{mainspathproducts}. Then $g_{\#}(F(\mathbb{Y},A)) \leq  H$.
\end{lemma}

\begin{proof}
Recalling that $f_{\#}(W_A)\leq H$, it suffices to show $g_{\#}(F(\mathbb{Y},A))\leq f_{\#}(W_A)$. To prove this inclusion, we identify a set of generators for $F(\mathbb{Y},A)$ using a few cases. Recall that $g\circ \rho_{I}=f\circ\upsilon_{[0,a]}$ whenever $I=(a,b)\in\mci(A)$ with $a>0$.

Case I: In the case that $0$ is an isolated point of $A$, there is a maximal $a\in (0,1]\cap A$ such that $[0,a]\cap A$ is finite. In this case, $\bbw_A\cap( [0,a]\times [0,1/2])$ is a finite graph which is the union of $[0,a]\times \{0\}$ and finitely many semicircles $C_I$, $I\in\mathcal{I}([0,a]\cap A)$. Thus $\pi_1(\bbw_A\cap ([0,a]\times [0,1/2]),w_0)$ is freely generated by $\{w_I\mid I\in \mathcal{I}([0,a]\cap A)\}$. Hence if $\alpha:\ui \to \bbw_A\cap ([0,a]\times [0,1/2])$ is a loop based at $w_0$, then $g_{\#}([\alpha])\in g_{\#}( i_{\#}(W_A))=f_{\#}(W_A)$.

Case II: Suppose $I=(a,b)\in\mathcal{I}(A)$ where $a>0$ and $b$ is not an isolated point of $A$. Then \[g_{\#}([\rho_I\cdot \upsilon_{[a,b]}\cdot \lambda_{[a,b]}^{-}\cdot \rho_{I}^{-}])=[(f\circ\upsilon|_{[0,b]})\cdot (f\circ\lambda_{[a,b]})^{-}\cdot (f\circ\upsilon|_{[0,a]})^{-}]=f_{\#}(w_I)\in f_{\#}(W_A).\]

Case III: Suppose $I=(a,b)\in\mathcal{I}(A)$ where $a>0$ and $b$ is an isolated point of $A$. Then we may write $J=(b,c)\in \mathcal{I}(A)$. We have $g_{\#}([\rho_I\cdot \upsilon_{[a,b]}\cdot \rho_{J}^{-}])=1$ and $g_{\#}([\rho_I\cdot \lambda_{[a,b]}\cdot \rho_{J}^{-}])=f_{\#}(w_{I}^{-1})\in f_{\#}(W_A)$.

Case IV: Fix (possibly equal) $I=(a,b)$ and $J=(c,d)$ in $\mci(A)$ where $A\cap[\min\{a,c\},\max\{b,d\}]$ is finite. Consider any reduced loop of the form $\alpha=\rho_{I}\cdot \beta\cdot \rho_{J}^{-}$ where $\beta$ is a finite concatenation of paths of the form $\upsilon_{K},\lambda_{K},\upsilon_{K}^{-}$, or $\lambda_{K}^{-}$ for $K\in\mci(A)$. Observe that $\alpha$ is homotopic to a finite concatenation of loops from either Case II or Case III. Hence, $g_{\#}([\alpha])$ factors as a product of elements in the subgroup $ f_{\#}(W_A)$ and must therefore be in that subgroup.

For the general case, let $\alpha:\ui\to \mathbb{Y}$ be a loop based at $w_0$ such that $[\alpha]\in F(\mathbb{Y},A)$. Since $\mathbb{Y}$ is one-dimensional, we may assume $\alpha$ is reduced and, thus, that $\alpha^{-1}(A)$ is finite. If $a$ is a limit point of $A$, then no subpath of $\alpha$ can cross the line $x=a$ within $\bbw_A$. Therefore, $\alpha$ is a reparameterization of a finite concatenation $\prod_{i=1}^{m}\alpha_i$ of loops $\alpha_i:\ui\to\mathbb{Y}$ based at $w_0$ where either 
\begin{enumerate}
\item $\alpha_i$ is a reparameterization of a finite concatenation of paths of the form $\upsilon_{K},\lambda_{K},\upsilon_{K}^{-}$, or $\lambda_{K}^{-}$ for $K\in\mci(A)$ (this can only occur if $0$ is an isolated point of $A$) as described in Case I.
\item $\alpha_i$ is a reparameterization of $\rho_{I}\cdot \beta_i\cdot \rho_{J}^{-}$ for intervals $I=(a,b)$ and $J=(c,d)$ in $\mathcal{I}(A)$ (where the two are possibly equal) where $A\cap[\min\{a,c\},\max\{b,d\}]$ is finite and $\beta_i$ is a reparameterization of a finite concatenation of paths of the form $\upsilon_{K},\lambda_{K},\upsilon_{K}^{-}$, or $\lambda_{K}^{-}$ for $K\in\mci(A)$. In this situation, $\alpha_i$ is a loop described by Case IV.
\end{enumerate}
In any of these cases (for all $i$), we have $g_{\#}([\alpha_i])\in f_{\#}(W_A)$ and, thus, $g_{\#}([\alpha])\in f_{\#}(W_A)$.
\end{proof}

\begin{remark}\label{lastremark}
Finally, we note how the proof of Theorem \ref{overallthm} is established from the above results. Let ($\ast$) represent the property: ``the trivial subgroup of $\pionex$ is $(C,c_{\infty})$-closed." Theorem \ref{groupthm} gives the equivalence of ($\ast$) with (1) and (2) for arbitrary spaces and with (5) if $X$ is first countable. The last two equivalences also hold for arbitrary spaces: Lemma \ref{omegapathlemma} in the case $H=1$ is ($\ast$) $\Leftrightarrow$ (3) and Theorem \ref{mainspathproducts} in the case $H=1$ is ($\ast$) $\Leftrightarrow$ (4). 
\end{remark}

\section{A remark on dense products}

We conclude with a brief remark on products over non-scattered order types. In light of Theorem \ref{overallthm}, it is natural to ask if the well-definedness of products indexed over any countable linear order, including dense orders, is also equivalent to the homotopically Hausdorff property.

\begin{definition}
Let $X$ be a path-connected space. We say that $X$ has
\begin{enumerate}
\item \textit{well-defined transfinite $\pi_1$-products} if for any $x\in X$ and maps $f,g:(\bbh,b_0)\to (X,x)$ from the Hawaiian earring such that $f\circ\ell_n\simeq g\circ\ell_n$ for each $n\in\bbn$, then $f_{\#}=g_{\#}$.
\item \textit{well-defined transfinite $\Pi_1$-products} if for any paths $\alpha,\beta:\ui\to X$ that agree on a closed set $S\subset\ui$ containing $\{0,1\}$ and such that $\alpha|_{[a,b]}$ is path-homotopic to $\beta|_{[a,b]}$ for each component $(a,b)$ of $\ui\backslash S$, then $[\alpha]=[\beta]$ in $\Pi_1(X)$.
\end{enumerate}
\end{definition}

The above definitions appear in \cite{BFTestMap} as Definitions 3.17 and 7.1 respectively, where they are both characterized in terms of closure operators.

\begin{question}
If $X$ is a homotopically Hausdorff Peano continuum, must $X$ have well-defined transfinite $\pi_1$-products and/or well-defined transfinite $\Pi_1$-products?
\end{question}

\end{document}